\documentclass[12pt]{amsart}
\usepackage{amsmath,amssymb,amsbsy,amsfonts,amsthm,latexsym,mathabx,
            amsopn,amstext,amsxtra,euscript,amscd,stmaryrd,mathrsfs,
            cite,array,mathtools,enumerate,bm}

\usepackage{color}
\usepackage{float}
\usepackage{url}
\usepackage[colorlinks,linkcolor=blue,anchorcolor=blue,citecolor=blue,backref=page]{hyperref}

\begin{document}

\newtheorem{theorem}{Theorem}
\newtheorem{lemma}[theorem]{Lemma} 
\newtheorem{claim}[theorem]{Claim}
\newtheorem{cor}[theorem]{Corollary}
\newtheorem{prop}[theorem]{Proposition}
\newtheorem{rem}[theorem]{Remark}
\newtheorem{definition}{Definition}
\newtheorem{quest}[theorem]{Open Question}

\newtheorem{example}[theorem]{Example}

\numberwithin{equation}{section}
\numberwithin{theorem}{section}

\def \blambda{\bm{\lambda}}

 \newcommand{\F}{\mathbb{F}}
\newcommand{\K}{\mathbb{K}}
\newcommand{\C}{\mathbb{C}}
\newcommand{\N}{\mathbb{N}}
\newcommand{\D}[1]{D\(#1\)}
\def\scr{\scriptstyle}
\def\\{\cr}
\def\({\left(}
\def\){\right)}
\def\[{\left[}
\def\]{\right]}
\def\<{\langle}
\def\>{\rangle}
\def\fl#1{\left\lfloor#1\right\rfloor}
\def\rf#1{\left\lceil#1\right\rceil}
\def\le{\leqslant}
\def\ge{\geqslant}
\def\eps{\varepsilon}
\def\mand{\qquad\mbox{and}\qquad}
\def\ep{\mathbf{e}_p}
\def\e{\mathbf{e}}
\def\vec#1{\mathbf{#1}}

\def\newu{m}

\newcommand{\commA}[1]{\marginpar{%
\vskip-\baselineskip 
\raggedright\footnotesize
\itshape\hrule\smallskip\color{blue}A.: #1\par\smallskip\hrule}}

\newcommand{\commR}[1]{\marginpar{%
\vskip-\baselineskip 
\raggedright\footnotesize
\itshape\hrule\smallskip\color{red}R.: #1\par\smallskip\hrule}}

\newcommand{\commI}[1]{\marginpar{%
\vskip-\baselineskip 
\raggedright\footnotesize
\itshape\hrule\smallskip\color{magenta}I.: #1\par\smallskip\hrule}}

\newcommand{\commII}[1]{\marginpar{%
\vskip-\baselineskip 
\raggedright\footnotesize
\itshape\hrule\smallskip\color{green}I.: #1\par\smallskip\hrule}}

\newcommand{\Fq}{\mathbb{F}_q}
\newcommand{\Q}{\mathbb{Q}}
\newcommand{\cS}{\mathcal{S}}
\newcommand{\Fp}{\mathbb{F}_p}
\newcommand{\Z}{\mathbb{Z}}
\newcommand{\PP}{\mathbb{P}}
\newcommand{\Disc}[1]{\operatorname{Disc}\(#1\)}
\newcommand{\Res}[1]{\operatorname{Res}\(#1\)}

\def\cA{{\mathcal A}}
\def\cB{{\mathcal B}}
\def\cC{{\mathcal C}}
\def\cD{{\mathcal D}}
\def\cE{{\mathcal E}}
\def\cF{{\mathcal F}}
\def\cG{{\mathcal G}}
\def\cH{{\mathcal H}}
\def\cI{{\mathcal I}}
\def\cJ{{\mathcal J}}
\def\cK{{\mathcal K}}
\def\cL{{\mathcal L}}
\def\cM{{\mathcal M}}
\def\cN{{\mathcal N}}
\def\cO{{\mathcal O}}
\def\cP{{\mathcal P}}
\def\cQ{{\mathcal Q}}
\def\cR{{\mathcal R}}
\def\cS{{\mathcal S}}
\def\cT{{\mathcal T}}
\def\cU{{\mathcal U}}
\def\cV{{\mathcal V}}
\def\cW{{\mathcal W}}
\def\cX{{\mathcal X}}
\def\cY{{\mathcal Y}}
\def\cZ{{\mathcal Z}}

\def\fM{{\mathfrak M}}

\newcommand{\Nm}[1]{\mathrm{Norm}_{\,\F_{q^k}/\Fq}(#1)}

\def\Tr{\mbox{Tr}}
\newcommand{\rad}[1]{\mathrm{rad}(#1)}

\title[Fields Generated by 
Polynomials of Given Height]{Discriminants of Fields Generated by 
Polynomials of Given Height}

\author{Rainer Dietmann} 
\address{Department of Mathematics, Royal Holloway, University of London, Egham, Surrey, TW20 0EX, United Kingdom} 
\email{rainer.dietmann@rhul.ac.uk}
\author{Alina Ostafe} 
\address{Department of Pure Mathematics, University of New South Wales, 
Sydney, NSW 2052, Australia}
\email{alina.ostafe@unsw.edu.au}
\author{Igor E. Shparlinski} 
\address{Department of Pure Mathematics, University of New South Wales, 
Sydney, NSW 2052, Australia}
\email{igor.shparlinski@unsw.edu.au}

\begin{abstract}
We obtain upper bounds for the number of monic irreducible polynomials over $\Z$ of a fixed degree $n$ and a
growing height $H$ for which the field generated by one of its roots has a given discriminant.
We approach it via counting square-free parts of polynomial discriminants via two complementing approaches.
In turn, this leads to a lower bound on the number of distinct discriminants
of  fields generated by roots of polynomials of degree $n$ and height at most $H$. We also give an upper bound for the number of trinomials of bounded height with given square-free part of the discriminant, improving previous results of 
I.~E.~Shparlinski~(2010).\end{abstract}

\maketitle

\section{Introduction}

\subsection{Motivation and background}

For a positive integer $H$, we use $\cP_n(H)$ to denote the set 
of polynomials
\begin{align*}
\cP_n(H) = \{X^n+a_{n-1}X^{n-1} + \ldots+a_1&X+a_0 \in \Z[X]~:\\
& ~|a_0|, \ldots, |a_{n-1}| <H \}.
\end{align*}

Furthermore, we use $\cI_n(H)$ to denote the set of irreducible polynomials 
from  $\cP_n(H)$. It is useful to recall that 
$$
\# \cI_n(H) = 2^nH^{n} + O(H^{n-1}), 
$$ 
which follows immediately
from much more precise results of Chela~\cite{Chela},  Dietmann~\cite{Diet1,Diet2}
and Zywina~\cite{Zyw}. We also note that Bhargava~\cite{Bha} has recently established the celebrated 
{\it  van der Waerden conjecture\/} about Galois groups of polynomials from $\cI_n(H)$. 

For  an irreducible monic polynomial  $f\in \Z[X]$ we use
$\Delta(f)$ to denote the discriminant 
of the algebraic number field
$\Q(\alpha)$, where  $\alpha$ is a root of $f$
(clearly, for any  $f \in \cI_n(H)$ all such fields $\Q(\alpha)$ are isomorphic
and thus have the same discriminant). 

For an integer $\Delta$ we denote by $N_n(H,\Delta)$ the number of polynomials 
$f \in \cI_n(H)$ with $\Delta(f) = \Delta$.

We recall that various counting problems for discriminants of number 
fields have been studied in a number of works, see~\cite{BSW, BBP, ILOSS, Jones1, Jones2, JoWh, ElVe, LaRo}
and references therein. In particular, 
a remarkable result of Bhargava,  Shankar, and Wang~\cite{BSW}
gives an asymptotic formula 
for the density of polynomials with square-free discriminants, however their model of counting
is different from ours. 
In fact, it seems that   the function $N_n(H,\Delta)$  which is our main object 
of study, 
has never been investigated before. 

We derive our estimates from some counting results 
on  square-free parts of discriminants of the polynomials from $\cP_n(H)$. 
We recall that the square-free part $u$ of an integer  $k$ is defined by $k = uv^2$ where $v^2$ 
is the largest perfect square dividing $k$. In particular, $u$ has the same sign as $k$. 

We remark that, despite the recent progress in~\cite{BSW}, the problem of counting square-free  
discriminants of the 
polynomials from $\cP_n(H)$ still remains
open, unless one assumes the celebrated $ABC$-conjecture, see~\cite{Kedl,Poon}.
So, one can consider our result as a first approximation to the desired goal. 

Furthermore, some counting results about square-free parts of discriminants 
\begin{equation}
\label{eq:DiscTrin}
  \Delta_n(a,b)=(n-1)^{n-1} a^n+n^n b^{n-1}
\end{equation}
of trinomials $X^n + aX + b$ with $n \equiv 1 \pmod 4$
have been given
in~\cite{MMS} (conditionally under the $ABC$-conjecture) and in~\cite{Shp1} (unconditionally). 
Here we obtain a new bound, improving that of~\cite{Shp1} for a wide range of parameters.

In fact our bound is a combination of two results, which we use depending on the relative 
sizes of parameters. One result is obtained via the {\it determinant method\/} of Bombieri and
Pila~\cite{BoPi}, Heath-Brown~\cite{HB3} and Salberger~\cite{S},
as in the work of Dietmann~\cite{Diet2}.  The other one  is based 
on the  {\it square sieve\/}
of Heath-Brown~\cite{HB1}, see Section~\ref{sec:SqSieve}, combined with bounds on character sums 
with discriminants,  see   Lemma~\ref{lem:CharSum p}, which are better than 
those directly implied by the Weil bound (see, for example,~\cite[Theorem~11.23]{IwKow}). 
We believe such bounds can be of independent interest.

\subsection{Notation}

We recall that
the expressions $A \ll B$,  $B \gg A$ and $A=O(B)$ are each equivalent to the
statement that $|A|\le cB$ for some positive constant $c$. We use $o(1)$ to denote any expression
that tends to $0$ for a fixed $n$ and $H\to\infty$.

 Throughout the paper, 
the implied constants in these symbols may depend on the degree $n$
of the polynomials involved, and occasionally, when mentioned explicitly, 
on some other parameters. 

The letters $p$ and $q$ always denote prime numbers.

\subsection{Discriminants of general polynomials}

Our  main result is the following upper bound on $N_n(H,\Delta)$, 
which is obtained by a combination of various techniques. 

\begin{theorem}
\label{thm:NHDelta}
Let $\Delta$ be a non-zero integer and $n \ge 3$.
Then, uniformly over $\Delta$, 
if for the square-free part $u$ of $\Delta$ neither $|u|(n-1)^{n-1}$ nor
$|u|n^n$ is a square, then
\begin{equation}
\label{eq:nn-1 discr}
N_n(H,\Delta)\le  H^{n-2+\sqrt{2}+o(1)},
\end{equation}
otherwise
\begin{equation}
\label{eq:Any discr}
N_n(H,\Delta)\ll \begin{cases}H^{n-2n/(3n+3)}(\log H)^{(5n+1)/(3n+3)} & \text{if}\ n \ge 5\\
 H^{n-n/(2n-1)}(\log H)^{(3n-2)/(2n-1)} & \text{if}\ n =3, 4.
 \end{cases} 
\end{equation}
for any $\Delta$. 
\end{theorem}

We note that the bound~\eqref{eq:nn-1 discr} is better (when it applies) than~\eqref{eq:Any discr}
only for $n \le 7$.

Let 
$$
M_n(H,D)= \sum_{|\Delta| \le D}   N_n(H,\Delta).
$$ 

Given a real parameter $D\ge 1$,
using that there are $O(D^{1/2})$ values of $\Delta\le D$
with  the square-free part  $u$ satisfying $|u|(n-1)^{n-1}$ or
$|u|n^n$ being a square, we immediately obtain that uniformly over $D$, 
$$
M_n(H,D) \le H^{o(1)} 
 \begin{cases}D H^{n-2n/(3n+3) } & \text{if}\ n \ge 8,\\
  DH^{n-2+\sqrt{2} } + D^{1/2} H^{n-2n/(3n+3) } & \text{if}\   n =5,6,7,\\
 DH^{n-2+\sqrt{2}} + D^{1/2} H^{n-n/(2n-1)} & \text{if}\ n =3, 4.
 \end{cases} 
$$

However one can get a better result. 
\begin{theorem}
\label{thm:MHD}
Let $D\ge 1$ be an integer. Then, for $n \ge 5$, uniformly over $D$, we have
$$
M_n(H,D) \le D^{(3n+2)/(3n+3)}  H^{n(3n+1)/(3n+3)+o(1)} .
$$
\end{theorem}

 Clearly, Theorem~\ref{thm:MHD} is nontrivial (that is, improves the trivial bound  $M_n(H,D)\ll H^n$) 
 provided that $D \le H^{2n/(3n+2) - \varepsilon}$ for some fixed $\varepsilon>0$. In particular, we see that 
 almost all polynomials from 
$\cI_n(H)$ generate fields with discriminants of size at least $H^{2n/(3n+2)+o(1)}$.

We also note very recent results of Anderson,   Gafni,   Lemke Oliver,   Lowry-Duda,   Shakan, and  Zhang~\cite{AGLLSZ} 
about the arithmetic structure of discriminants of polynomials from $\cP_n(H)$.

\begin{rem}
\label{rem:disc split}  
We note that the bound of Theorem~\ref{thm:NHDelta}   also implies an
upper bound for the number $K_n(H,\delta)$ of polynomials $f\in\cI_n(H)$ such that the discriminant $\delta(f)$ of the splitting field $L_f$ of $f$ is $\delta$. Indeed, for a given $f\in \cI_n(H)$, we have the divisibility 
$\Delta(f)\mid\delta(f)$.  
Thus, using $\delta(f)  = H^{O(1)}$ for $f\in\cI_n(H)$ and the classical bound
$\tau(\delta) = \delta^{o(1)}$ for the divisor function $\tau$,  we obtain
$$
K_n(H,\delta)\le \sum_{\Delta \mid\delta}
  N_n(H, \Delta) \le  H^{o(1)}  \begin{cases}H^{n-2n/(3n+3) }  & \text{if}\ n \ge 5,\\
 H^{n-n/(2n-1) }  & \text{if}\ n =3, 4.
 \end{cases} 
$$
Similarly, we also have  an analogue of Theorem~\ref{thm:MHD}  for $K_n(H,\delta)$. 
\end{rem}

\subsection{Discriminants of trinomials}
\label{sec:discr trinom}

Let $T_n(A,B,C,D;u)$ be the number of pairs of integers
$(a,b) \in [C,C+A] \times [D,D+B]$ such that for the trinomial discriminant~\eqref{eq:DiscTrin}
we have 
$\Delta_n(a,b) =ur^2$,
for some positive integer $r$. 
 
 For $n \equiv 1 \pmod 4$,
$A \ge 1$, $B \ge 1$, $C \ge 0$, $D \ge 0$ and square-free $u$, Shparlinski~\cite[Theorem~1]{Shp1} has obtained the bound
\begin{align*}
  T_n(A,B,C,D;u) & \ll (AB)^{2/3} (\log (AB))^{4/3} +
  (A+B) (\log(AB))^2\\
  & \qquad \qquad \qquad + (AB)^{1/3}
  \left( \frac{\log(ABCD) \log (AB)}{\log \log (ABCD)} \right)^2,
\end{align*}
using exponential sums and the square sieve. For $C \ge 1$ and
$A \ll B^{2-\varepsilon}$ with an arbitrary fixed $\varepsilon > 0$, we can sharpen this as follows.

\begin{theorem}
\label{thm:trinom}
Let $n \equiv 1 \pmod 4$, $n \ge 2$,
$A \ge 1$, $B \ge 1$, $C \ge 1$, $D \ge 0$,
and let $u$ be square-free. Then
$$
  T_n(A,B,C,D;u) 
  \le A (A+B+C+D)^{o(1)}.
$$
\end{theorem}

As in~\cite{Shp1}, from this we obtain the following two results:

\begin{cor}
\label{cor:Sn}
In the notation of Theorem~\ref{thm:trinom}, let $S_n(A,B,C,D)$ be the
number of distinct quadratic fields $\Q(\sqrt{\Delta_n(a,b)})$
taken for all pairs of integers $(a,b) \in [C,C+A] \times
[D,D+B]$ such that $X^n+aX+b$ is irreducible over $\Q$.
Then under the assumptions of Theorem~\ref{thm:trinom}, we have
$$
  S_n(A,B,C,D) \ge B (A+B+C+D)^{o(1)}.
$$
\end{cor}

In fact, in Corollary~\ref{cor:Sn}, the lower bound holds for the number of distinct 
square-free parts of discriminants $\Delta_n(a,b)$. Thus, taking $C=D=1$ 
and $A=B=H$ in Corollary~\ref{cor:Sn}, by Lemma~\ref{lem:D/Delta} below, we also obtain the following:

\begin{cor}
\label{cor:distinct disc} 
For $n \ge 2$ with $n \equiv 1 \pmod 4$, the number of distinct discriminants of fields generated by a root of polynomials from
$$
  \{X^n+aX+b~:~ 1\le a,b \le H\}
$$
is at least $H^{1+o(1)}$.
\end{cor} 

\begin{cor}
\label{cor:QX}
Let $Q_n(\Delta)$ be the number of distinct quadratic fields $\Q(\sqrt{\Delta_n(a,b)})$
taken over all integers $a,b \ge 1$ such that $X^n+aX+b$ is irreducible over
$\Q$ and $|\Delta_n(a,b)| \le \Delta$. Then, for $n \equiv 1 \pmod 4$ we have
$$
  Q_n(\Delta) \ge \Delta^{1/(n-1)+o(1)}.
$$ 
\end{cor}

Corollary~\ref{cor:QX} improves the bound
$$
  Q_n(\Delta) \gg \Delta^{\kappa_n/3} (\log \Delta)^{-1}
$$
in~\cite{Shp1}, where
$$
  \kappa_n = \frac{1}{n} + \frac{1}{n-1},
$$
by asymptotically a factor $3/2$ in the exponent.

We conclude the paper with an appendix  where we take the opportunity to correct an error in~\cite[Lemmas~5 and~6]{Diet2}  (and consequently~\cite[Lemma~8]{Diet2}) which are not correct as stated if the degree $n$ is of the form
$n=\newu^2$ or $n=\newu^2+1$ for some odd $m$, see Section~\ref{sec:app}.

\section{Preparations}

\subsection{Polynomials and discriminants}

We recall that for an  
arbitrary field $\K$ and $f=X^n+a_{n-1}X^{n-1} +\ldots+a_1X+a_0\in\K[X]$, the discriminant 
of $f$ is defined by
\begin{equation}
\label{eq:D vs R}
\Disc{f}=(-1)^{n(n-1)/2}\Res{f,f'},
\end{equation}
where $\Res{g,h}$ denotes the resultant of $g,h \in \K[X]$. 

Throughout we treat $\Disc{f}$ as a polynomial in formal 
variables $a_0, \ldots, a_{n-1}$. 

It is well-known, see~\cite[Section~3.3]{Lang}, that $\Disc{f}$ and $\Delta(f)$
are related via an integer square.

\begin{lemma}
\label{lem:D/Delta} Let $f\in\Q[X]$ be a monic irreducible polynomial. Then  $\Disc{f}/\Delta(f)=r^2$ for some integer $r\ge 1$. 
\end{lemma}

We also recall that the question about the number of polynomials 
$f \in \cI_n(H)$ with $\Delta(f) = \Disc{f}$ remains unanswered. 
Ash, Brakenhoff and Zarrabi~\cite{ABZ} give some heuristic
and numerical evidences towards the conjecture, attributed in~\cite{ABZ} to
Hendrik Lenstra, that the density of such polynomials is $6/\pi^2$. 
We remark that this density is higher than the expected density of 
square-free discriminants $\Disc{f}$ (in which case we immediately 
obtain $\Delta(f) = \Disc{f}$ by Lemma~\ref{lem:D/Delta}), see~\cite{ABZ}
for a discussion of this phenomenon. 

We now need several results about the irreducibility of some polynomials  
involving polynomial discriminants. For the rest of the paper the discriminant $\Disc{F}$ of a polynomial $F$ always means the discriminant with respect to the variable $X$, even if 
the polynomial $F$ may depend on other variables. 

\begin{lemma}
\label{l1}
Let $n \ge 3$, let $a_{2}, \ldots, a_{n-1} \in \Z$ and let
$c_0, c_1 \in \Q$. Moreover, let $u \in \Z$ be square-free such that
neither $|u|(n-1)^{n-1}$ nor $|u|n^n$ is a square.
Then the polynomial
\begin{align*}
  Z^2-u \Disc{X^n+a_{n-1} X^{n-1} + \ldots + a_2 X^2
  + (c_0 A_0 + c_1) X + A_0}\quad &\\
   \in \Q[A_0,Z] &
\end{align*}
is irreducible in $\Q[A_0,Z]$.
\end{lemma}

\begin{proof} 
We closely follow the proof of~\cite[Lemma~5]{Diet2}, see also Section~\ref{sec:app}. Writing
$$
  D(A_0) = u \Disc{X^n+a_{n-1} X^{n-1} + \ldots + a_2 X^2 + (c_0 A_0 + c_1) X + A_0},
$$
it is enough to show that $D(A_0)$ is no square in $\Q[A_0]$.

For $c_0 \ne 0$, by~\cite[Lemma~4]{Diet2},
we find that the monomial in $D(A_0)$ with biggest degree is
$$
  u (-1)^{(n-1)(n-2)/2} (n-1)^{n-1} c_0^n A_0^n,
$$
which cannot be a square in $\Q[A_0]$. Indeed, if $n$ is odd this is obvious. If $n$ is even this is true
since $|u|(n-1)^{n-1}$ is not a square but $c_0^n$ is.

For $c_0=0$, by~\cite[Lemma~3]{Diet2}, one finds that the monomial in
$D(A_0)$ with biggest degree is
$$
  u (-1)^{n(n-1)/2} n^n A_0^{n-1}.
$$
Again, since $|u|n^n$ is no square, this cannot be a square in $\Q[A_0]$.
\end{proof}

In the same way one proves the following analogue of~\cite[Lemma~6]{Diet2}.

\begin{lemma}
\label{l2}
Let $n \ge 3$, let $a_2, \ldots, a_{n-1} \in \Z$ and $c \in \Q$.
Moreover, let $u \in \Z$ be square-free such that $|u|(n-1)^{n-1}$
is not a square. Then the polynomial
$$
  Z^2 - u \Disc{X^n+a_{n-1} X^{n-1} + \ldots +
  a_2 X^2 + A_1 X + c}\in \Q[A_1,Z]
$$
is irreducible in $\Q[A_1,Z]$.
\end{lemma}

The argument below is modelled from that of the proof of~\cite[Lemma~4]{Shp0}.

For a monic polynomial $f(X)\in\K[X]$ and $(u,v) \in \K^*\times \K$, 
we define the polynomial 
$$
f_{u,v}(X) = u^{n} f\(u^{-1} \(X+v\)\) \in\K[X],
$$
which we  write as
\begin{equation}
\label{eq:Auv def}
f_{u,v}(X) = X^n + \sum_{j=1}^{n}A_{f,j} (u,v) X^{n-j}.
\end{equation}
One easily verifies that by the Taylor formula
\begin{equation}
\label{eq:Auv}
A_{f,j} (u,v) = u^j \frac{f^{(n-j)}(u^{-1} v)} {(n-j)!}, \qquad j =1, \ldots, n.
\end{equation}

We need some simple properties of the polynomials $f_{u,v}(X)$.
First we relate $\Disc{f_{u,v}}$ to $\Disc{f}$. The following statement is 
shown in the proof of~\cite[Theorem~1]{Shp2};
 it follows easily via the standard expression 
of the discriminant via the roots of the corresponding polynomial and the relation 
between the roots of $f$ and $f_{u,v}$.

\begin{lemma}
\label{lem:Discr uv} For any field $\K$ and a monic polynomial 
$f(X) \in \K[X]$ of degree $n$, we have
$$
\Disc{f_{u,v}} = u^{n(n-1)}\Disc{f}, \qquad (u,v) \in \K^*\times \K.
$$
\end{lemma}

Let $\F_p$ denote the finite field of $p$ elements. 

We now show that the map $f \mapsto f_{u,v}$ is almost a permutation on 
the set of monic polynomials   $f(X) \in \F_p[X]$ of fixed degree.

\begin{lemma}
\label{lem:Distinct} For a prime $p> n$, for all but at most $O\(p^{\fl{n/2}+1}\)$ monic polynomials  
$f(X) \in \F_p[X]$ of degree $n$, 
 the polynomials $f_{u,v}(X)$, $(u,v) \in \F_p^*\times \F_p$, are pairwise distinct. 
\end{lemma}

\begin{proof}  
Let $\cA$ be the set of  $m \le n $ distinct roots of $f$. Then the non-uniqueness condition 
$$
f_{s,t}(X)= f_{u,v}(X)
$$ 
with $(s,t) \ne (u,v)$ means that for any $\alpha \in \cA$ there is $\beta \in \cA$ 
with $s\alpha  - t = u\beta - v$. Hence there is a nontrivial linear transformation 
$\cA \mapsto a\cA + b$, sending each element $\alpha \in \cA$ to $a\alpha+b$, which fixes the set $\cA$, that is,
$$
\cA = a\cA + b.
$$

If $a=1$ then $b \ne 0$ and examining the orbit 
$$
\alpha  \mapsto \alpha  + b  \mapsto \alpha +2b\mapsto \ldots
$$
 of any element $\alpha \in \cA$ we see that for some 
$k \le n$ we have to have $\alpha = \alpha + kb$ which is impossible since $p > n$.

Assume now that $a\ne 1$.
Hence for $\cB = \cA +b(a-1)^{-1}$ we have  
\begin{equation}
\label{eq:Set B}
\cB = a\cB.
\end{equation}
Examining the orbit 
$$
\beta \mapsto a\beta  \mapsto a^2 \beta  \mapsto \ldots
$$
of any non-zero element $\beta \in \cB$ we see that 
$a$ is of multiplicative order at most $m$ and thus takes at most $m(m+1)/2$ 
possible values. 

Finally, when $a\ne 1$ is fixed, there are at most $O\(p^{\fl{n/2}}\)$ possibilities 
for the set $\cB$. Indeed, we see from~\eqref{eq:Set B} that  $\cB$  is a union of cosets 
of the multiplicative group $\langle a\rangle \subseteq \F_p^*$ generated by $a$, and possibly of $\{0\}$. Since $a\ne 1$ 
we see that $\#\langle a\rangle \ge 2$ so  each such coset if of size at least $2$, and thus there are 
at most  $\fl{m/2}$ such cosets in $\cB$.
We now observe that, for a fixed 
$a$, any such coset  is defined by any of its elements.  Hence the number of possibilities for the set 
$\cB$ does not exceed the number of choices of $\fl{m/2}$ distinct elements of $\F_p$.

When the set $\cB$ is fixed, there are $p$ possibilities for $b \in \F_p$. Therefore we conclude that there are $O\(p^{\fl{m/2}+1}\)$
possibilities for the set of roots $\cA$. Since there are $O(1)$ choices for the multiplicities of these roots,
and $m \le n$, the result follows. 
\end{proof}

\subsection{Character sums with discriminants}

We remark that the values of the quadratic character $\chi$ of discriminants are polynomial 
analogues of the M{\"o}bius functions for integers, since by 
the Stickelberger theorem~\cite{Dalen,St}, 
for a square-free polynomial $f\in\F_p[X]$ of degree $n$, where
$p$ is odd, 
\begin{equation}
\label{eq:Stickel}
 \(\frac{\Disc{f}}{p}\)=(-1)^{n-r},
\end{equation}
where  $(u/p)$ is the Legendre symbol of $u$ modulo $p$ and  
 $r$ is the number of distinct irreducible factors of $f$
and, of course, 
$$ \(\frac{\Disc{f}}{p}\)= 0
$$ 
if $f$ is not square-free. 
In particular, this interpretation has motivated the work of  
Carmon and Rudnick~\cite{CaRu}. Here we also 
need some simple estimates.

Let $\cM_{n,p}$ be the set of monic polynomials of degree $n$ over $\F_p$. 

\begin{lemma}
\label{lem:sumDisc} For a prime $p\ge 3$,
$$
\sum_{f \in \cM_{n,p}}  \(\frac{\Disc{f}}{p}\)=0.
$$
\end{lemma}

\begin{proof}
Let $\cJ(p)$ be the set of all monic irreducible polynomials over $\F_p$.  
We consider the zeta function $\zeta(T)$ of the affine line over $\F_p$, which is given by
the product
$$
\zeta(T)= \prod_{g \in\cJ(p)} \(\frac{1}{1-T^{\deg  g}}\)
= \frac{1}{1-pT},
$$
that is absolutely converging for $|T| < 1$,
see~\cite[Equations~(1) and~(2)]{Ros}.
Taking the inverse, we derive
\begin{equation}
\label{eq:zetainv}
\zeta(T)^{-1}=(1-pT)=\prod_{g \in\cJ(p)}  \(1-T^{\deg  g}\)
 =\sum_{ f\in\cS(p)} (-1)^{\omega(f)}T^{\deg f},
\end{equation}
where $\cS(p)$ is the  set of all monic square-free polynomials  over $\F_p$
and $\omega(f)$ denotes the number of distinct irreducible factors of $f$.

Using~\eqref{eq:Stickel} and comparing the coefficient of $T^n$ in the equation~\eqref{eq:zetainv}, we obtain  the claimed result.
\end{proof}

Now,  for a vector 
$$
\blambda = \(\lambda_1 ,  \ldots, \lambda_n\) \in \F_p^n 
$$
and a polynomial  
$$f(X)  = X^n+a_{n-1}X^{n-1} + \ldots+a_1X+a_0\in  \cM_{n,p}
$$
we define
\begin{equation}
\label{eq:InnProd}
\langle \blambda \circ f\rangle = \lambda_1 a_{n-1}  + \ldots+ \lambda_{n} a_{0}.
\end{equation}

For an integer $m\ge 1$, we denote
$$
\e_m(z) = \exp(2 \pi i z/m),
$$
and  consider certain mixed exponential and character sums with polynomials.
Bounds of these sums underly our approach via the square-sieve method. 

We emphasise that our bounds in Lemmas~\ref{lem:CharSum p large n} and~\ref{lem:CharSum p}   below save $\max\{p^{(n-1)/4}, p\}$ against the trivial bound, while an immediate
application of the 
classical  Weil bound (see, for example,~\cite[Theorem~11.23]{IwKow}) saves only $p^{1/2}$.  The existence 
of such a bound is quite remarkable since 
the discriminant  $\Disc{f}$, as a polynomial in the coefficients of $f$, 
is highly singular: its locus of singularity is of 
co-dimension one, see~\cite[Section~4]{Shp2}. In particular, this means that the result of
Katz~\cite{Katz} does not apply, while the result of Rojas-Le{\'o}n~\cite{Ro-Le} does not give any advantage 
over the direct application of the Weil bound~\cite[Theorem~11.23]{IwKow}, which saves $p^{1/2}$
over the trivial bound. Instead we recall the following bound obtained independently by
Bienvenu and  L{\^e}~\cite[Theorem~1]{BiLe} and Porritt~\cite[Theorem~1]{Por}. 

\begin{lemma}
\label{lem:CharSum p large n} 
Let $p \ge 3$ be a  prime. 
Then, for $n \ge 3$ and  any $\blambda   \in \F_p^n $, in the notation~\eqref{eq:InnProd},  we have
$$
\sum_{f \in \cM_{n,p} }  \(\frac{\Disc{f}}{p}\)
\ep\(\langle \blambda \circ f\rangle \)
\ll p^{(3n+1)/4}. 
$$
\end{lemma}

We now use a different argument, which stems from~\cite{Shp0}, to get a larger saving for small values of $n$.

\begin{lemma}
\label{lem:CharSum p} 
Let $p \ge 3$ be a  prime. 
Then, for $n \ge 3$ and  any $\blambda   \in \F_p^n $, in the notation~\eqref{eq:InnProd},  we have
$$
\sum_{f \in \cM_{n,p} }  \(\frac{\Disc{f}}{p}\)
\ep\(\langle \blambda \circ f\rangle \)
\ll p^{n-1}. 
$$
\end{lemma}

\begin{proof} By Lemma~\ref{lem:sumDisc} we can assume that $\blambda$ is not identical to zero. 

Let $\cE_{n,p}$ be the
exceptional set of polynomials which are described in Lemma~\ref{lem:Distinct}, 
that is, the set of monic polynomials  
$f(X) \in \F_p[X]$ of degree $n$, such that
 the polynomials $f_{u,v}(X)$, $(u,v) \in \F_p^*\times \F_p$, are not pairwise distinct.

 Thus,  by Lemma~\ref{lem:Distinct}, for $n \ge 3$, we have 
 \begin{equation}
\label{eq:set E}
\# \cE_{n,p} = O\(p^{\fl{n/2}+1}\) =  O\(p^{n-1}\).
\end{equation}

For  $f \in \cM_{n,p}$ we define the quantity $R(f)$ as the following product 
of the resultants of the consecutive derivatives of $f$: 
$$
R(f) = \prod_{j = 0}^{n-1}\Res{f^{(j)},f^{(j+1)}}. 
 $$
 Let $\cF_{n,p}$ be the set of $f \in \cM_{n,p}$ with $R(f) = 0$. 
 Clearly 
 \begin{equation}
\label{eq:set F} 
\# \cF_{n,p} = O\(p^{n-1}\).
\end{equation}
Define 
$$
\cL_{n,p} = \cM_{n,p} \setminus \(\cE_{n,p}\cup \cF_{n,p}\). 
$$

We now see from~\eqref{eq:set E} and~\eqref{eq:set F} that 
 for any  $(u,v) \in \F_p^*\times \F_p$ we have 
\begin{equation}
\begin{split}
\label{eq:M2L}
\sum_{f \in \cM_{n,p} }  & \(\frac{\Disc{f}}{p}\)  
\ep\(\langle \blambda \circ f\rangle \) \\
& = \sum_{f \in \cL_{n,p} }   \(\frac{\Disc{f}}{p}\)
\ep\(\langle \blambda \circ f\rangle \) + O\(p^{n-1}\)\\
& = \sum_{f \in \cL_{n,p} }   \(\frac{\Disc{f_{u,v}}}{p}\)
\ep\(\langle \blambda \circ f_{u,v}\rangle \) + O\(p^{n-1}\). 
\end{split}
\end{equation}
Since $n(n-1)$ is even, by Lemma~\ref{lem:Discr uv} 
 we have 
 $$
  \(\frac{\Disc{f_{u,v}}}{p}\) =  \(\frac{\Disc{f}}{p}\). 
$$
Thus, summing~\eqref{eq:M2L} over all pairs  $(u,v) \in \F_p^*\times \F_p$
and changing the order of summation, we obtain 
\begin{align*}
\sum_{f \in \cM_{n,p} }  & \(\frac{\Disc{f}}{p}\)  
\ep\(\langle \blambda \circ f\rangle \) \\
& =\frac{1}{p(p-1)}\sum_{f \in \cL_{n,p} }   \(\frac{\Disc{f}}{p}\)
\sum_{(u,v) \in \F_p^*\times \F_p} \ep\(\langle \blambda \circ f_{u,v}\rangle \) \\
& \qquad \qquad  \qquad \qquad  \qquad \qquad  \qquad \qquad  \qquad  + O\(p^{n-1}\). 
\end{align*}
Extending the summation to all pairs $(u,v) \in \F_p\times \F_p$ introduces an error $ O\(p^{n-1}\)$
which is admissible, so we obtain 
\begin{equation}
\begin{split}
\label{eq:M2L uv}
\sum_{f \in \cM_{n,p} }  & \(\frac{\Disc{f}}{p}\)  
\ep\(\langle \blambda \circ f\rangle \) \\
& =\frac{1}{p(p-1)}\sum_{f \in \cL_{n,p} }   \(\frac{\Disc{f}}{p}\)
\sum_{(u,v) \in \F_p^2} \ep\(\langle \blambda \circ f_{u,v}\rangle \) \\
& \qquad \qquad  \qquad \qquad  \qquad \qquad  \qquad \qquad  \qquad  + O\(p^{n-1}\). 
\end{split}
\end{equation}

Recalling the notation~\eqref{eq:Auv def}
we write 
$$
\sum_{(u,v) \in \F_p^2} \ep\(\langle \blambda \circ f_{u,v}\rangle \)  
= \sum_{(u,v) \in \F_p^2} \ep\(\sum_{j=1}^{n} \lambda_j A_{f,j} (u,v)  \) .
$$
We now see from~\eqref{eq:M2L uv} that it is enough to show that for any $f \in \cL_{n,p}$ 
the Deligne bound (see~\cite[Section~11.11]{IwKow}) applies to the last sum and thus implies the 
bound 
\begin{equation}
\label{eq:Deligne}
 \sum_{(u,v) \in \F_p^2} \ep\(\sum_{j=1}^{n} \lambda_j A_{f,j} (u,v)  \)  = O(p).
\end{equation}
For this we have to show that the highest form of the polynomial 
$$
F_f(U,V) = \sum_{j=1}^{n} \lambda_j A_{f,j} (U,V) \in \F_p[U,V]
$$
is nonsingular.   Since $\blambda$ is not identical to zero
there is $m$ such that  $\lambda_m \ne 0$ and $\lambda_j = 0 $ for $j > m$
(this condition is void if $m=n$). 

We see from~\eqref{eq:Auv} that $A_{f,j} (U,V) $ is a
homogeneous polynomial of degree 
$\deg A_{f,j} (U,V)  = j$, $ j=1, \ldots, n$.
Hence the highest form of $F_f(U,V)$ is $\lambda_m  A_{f,m} (U,V)$. 
Therefore, to establish the bound~\eqref{eq:Deligne},  it is sufficient to show that the
polynomial $\lambda_m A_{f,m} (U,V) $ is nonsingular, that is, that the 
 equations
\begin{equation}
\label{eq:Eq dU}
\frac{ \partial A_{f,m} (U,V)}{\partial U} = 0
\end{equation}
and 
\begin{equation}
\label{eq:Eq dV}
\frac{ \partial A_{f,m} (U,V)}{\partial V} = 0
\end{equation}
have no common zero $(u_0,v_0) \ne (0,0)$  in the algebraic closure of $\F_p$. 

Now, if $u_0 = 0$, then from~\eqref{eq:Eq dV} we conclude that $v_0=0$, which is 
impossible. Indeed, from~\eqref{eq:Auv} we see that 
$$
A_{f,m} (U,V)  \equiv \binom{n}{m}  V^m  \pmod U 
$$
in the ring $\F_p[U,V]$. Hence 
$$
\frac{ \partial A_{f,m} (U,V)}{\partial V}  \equiv \binom{n}{m} m V^{m-1}  \pmod U, 
$$
which implies the above claim. 

If  $u_0 \ne 0$, then by  the Euler formula for partial derivatives we have
$$
U \frac{ \partial A_{f,m} (U,V)}{\partial U} +  V \frac{ \partial A_{f,m} (U,V)}{\partial V} 
= m U^m \frac{f^{(n-m)}(U^{-1} V)} {(n-m)!} .
$$
Therefore, we conclude from the equations~\eqref{eq:Eq dU} and~\eqref{eq:Eq dV}  
that $v_0/u_0$ is a zero of $f^{(n-m)}(X)$, 
and also by~\eqref{eq:Eq dV},  $v_0/u_0$ is also a zero of $f^{(n-m+1)}(X)$. Hence 
$\Res{f^{(n-m)}, f^{(n-m+1)}} = 0$, which contradicts the condition that $f \not \in \cF_{n,p}$. 
Thus, we have the bound~\eqref{eq:Deligne} and the desired result follows. 
 \end{proof}

We now recall the definition of the Jacobi symbol $(u/m)$
modulo an odd square-free integer $m$:
$$
\(\frac{u}{m}\) = 
\prod_{\substack{p \mid m\\ p~\mathrm{prime}}} \(\frac{u}{p}\) ,
$$
where, as before,  $\(\frac{u}{p}\)$ is the Legendre symbol (that is, the
quadratic character) modulo a prime $p$, see~\cite[Section~3.5]{IwKow}).

We not extend the definition of  $\cM_{n,p}$ to residue rings, and use  $\cM_{n,m}$  to denote
 the set of monic polynomials of degree $n$ over $\Z_m$. 

Now, using the Chinese Remainder Theorem for 
character sums, see~\cite[Equation~(12.21)]{IwKow}, 
we see that  Lemma~\ref{lem:sumDisc}, implies the following identity.

\begin{lemma}
\label{lem:CharSum m compl} 
Let $p$ and $q$  be two sufficiently large distinct primes and let 
$m=pq$.
Then, for $n \ge 3$  we have
$$
\sum_{f \in \cM_{n,m} }\(\frac{\Disc{f}}{m}\) = 0. 
$$
\end{lemma}

Similarly, we see that  Lemmas~\ref{lem:CharSum p large n} and~\ref{lem:CharSum p}, 
together  with the Chinese Remainder Theorem for 
mixed sums of additive and multiplicative 
characters, see~\cite[Equation~(12.21)]{IwKow}, 
yield   the following bound.

\begin{lemma}
\label{lem:CharSum m} 
Let $p$ and $q$  be two sufficiently large distinct primes and let 
$m=pq$.
Then, for $n \ge 3$ and  any  $\blambda   \in \Z_m^n$, in the notation~\eqref{eq:InnProd},  we have
$$
\sum_{f \in \cM_{n,m} }\(\frac{\Disc{f}}{m}\)
\e_m\(\langle \blambda \circ f\rangle \)
\ll \min\left\{m^{(3n+1)/4}, m^{n-1}\right\}. 
$$
\end{lemma}

We now derive our main tool.

\begin{lemma}
\label{lem:CharSum m H} 
Let $p$ and $q$  be two sufficiently large distinct primes and let 
$m=pq$.
Then  we have
\begin{align*}
\sum_{f \in \cP_n(H)}& \(\frac{\Disc{f}}{m}\)\\
&\quad \ll \(\(H/m\)^{n-1}\log m  +\(\log m\)^{n}\) 
\begin{cases}m^{(3n+1)/4} & \text{if}\ n \ge 5.\\
m^{n-1}& \text{if}\ n =3,4.
\end{cases} 
\end{align*}
\end{lemma}

\begin{proof}
Clearly the above sum can be split into $O(H^n/m^n)$ complete sums,
which all vanish by Lemma~\ref{lem:CharSum m compl}, 
and also, for $k=0, \ldots, n-1$ into $O(H^{k}/m^{k}+1)$ hybrid  sums, that are 
complete with respect to exactly $k$ variables  and incomplete with respect 
to the remaining $n-k$ variables. 
Using the standard reduction between complete and
incomplete sums (see~\cite[Section~12.2]{IwKow}) and applying  
Lemma~\ref{lem:CharSum m}  (for  incomplete sums), we derive
\begin{align*}
\sum_{f \in \cP_n(H)}& \(\frac{\Disc{f}}{m}\)\\
&\quad\ll \sum_{k=0}^{n-1} (H^{k}/m^{k}+1) \min\left\{m^{(3n+1)/4}, m^{n-1}\right\} (\log m)^{n-k}, 
\end{align*}
which implies the result. 
\end{proof}

\section{Proof of Theorem~\ref{thm:NHDelta}}

\subsection{The bound~\eqref{eq:nn-1 discr}: the determinant method}
\label{sec:DetMethod}
We need  some  results about equations  involving discriminants,
which could be of  independent interest. 

\begin{lemma}
\label{l3}
Let $n \ge 3$, let
$a_2, \ldots, a_{n-1} \in \Z$ and let $d_0, d_1, d_2 \in \Q$ such
that $(d_0, d_1) \ne (0, 0)$.
Moreover, let $u \in \Z$ be square-free such that
neither $|u|(n-1)^{n-1}$ nor $|u|n^n$ is a square.
Then, for any $c \ge 1$, the system of equations
\begin{align*}
   z^2 & = u\Disc{X^n + a_{n-1} X^n + \ldots + a_1 X + a_0}\\
   0 & = d_0 a_0 + d_1 a_1 + d_2
\end{align*}
has at most $ H^{1/2+o(1)}$ solutions
$z, a_0, a_1 \in \Z$ such that 
$$
|a_0|, |a_1| \le H \mand |z| \le H^c.
$$
\end{lemma}

\begin{proof}
This is a straightforward generalization of~\cite[Lemma~8]{Diet2}, which
dealt with the special case $u=1$. Lemmas~\ref{l1} and~\ref{l2} 
now play the role of~\cite[Lemma~5]{Diet2} and~\cite[Lemma~6]{Diet2},
respectively. The proof can then be followed in a completely analogous
way to the proof of~\cite[Lemma~8]{Diet2}.
\end{proof}

We also need the following technical result, which generalises~\cite[Lemma~11]{Diet2}.

\begin{lemma}
\label{l4}
Let $u \in \Z \backslash\{0\}$, and let $N(H)$ be the number of coefficients 
$a_2, \ldots, a_{n-1} \in \Z$ such that
$|a_i| \le H \quad (2 \le i \le n-1)$ and the polynomial
\begin{equation}
\begin{split}
\label{eq:polyAAZ}
  Z^2-u \Disc{X^n+a_{n-1} X^{n-1} + \ldots +a_{2} X^2 + A_1 X  + A_0}\qquad &\\
  \in \Z[A_0, A_1, Z]&
\end{split}
\end{equation}
as a polynomial in $A_0, A_1, Z$ is not absolutely irreducible.
Then
$$
  N(H) \ll  H^{n-3}.
$$
\end{lemma}

\begin{proof}
The special case $u=1$ is just~\cite[ Lemma~11]{Diet2}. However, if
the polynomial~\eqref{eq:polyAAZ}  factorises over $\C[A_0, A_1, Z]$,
then 
\begin{equation}
\label{eq:polyAA}
u\Disc{X^n+a_{n-1} X^{n-1} + \ldots +a_{2} X^2 + A_1 X  + A_0}\in \C[A_0, A_1]
\end{equation}
is a perfect square in $\C[A_0, A_1]$ whence, as $u \ne 0$,
also the polynomial~\eqref{eq:polyAA} with $u=1$ is a square in
$\C[A_0, A_1]$. Hence also the polynomial~\eqref{eq:polyAAZ}  with $u=1$
factorises
over $\C[A_0, A_1, Z]$. The result therefore follows immediately from
the special case $u=1$.
\end{proof}

Given a square-free integer $u\ge 1$, we denote by $\cT_n(H,u)$ the set of $f \in \cI_n(H)$  for which the square-free part of $\Delta(f)$  is $u$, 
that is, $|\Delta(f)| = r^2u$ for some integer $r\ge 1$, and by
$T_n(H,u)$ the cardinality of this set. 

\begin{lemma}
\label{lem:det}
Uniformly over   square-free integers $u \ge 1$ with the condition that
neither $u(n-1)^{n-1}$ nor $un^n$ is a square,
we have the following estimate
$$
T_n(H,u) \le H^{n-2+\sqrt{2}+o(1)}.
$$
\end{lemma}
 
\begin{proof} 
Let us fix some $\varepsilon > 0$. 
By Lemma~\ref{lem:D/Delta}, $\Disc{f}$  and   $\Delta(f)$ 
have the same  square-free part. 
So, for square-free $u \in \Z$,  we see that $T_n(H, u)$ is the number of solutions
$a_0, \ldots, a_{n-1} \in \Z, r \in \N$ of the Diophantine equation
\begin{equation}
\label{aug6}
  r^2 u = \Disc{X^n + a_{n-1} X^{n-1} + \ldots +
  a_1 X + a_0}
\end{equation}
such that $|a_i| \le H \quad (0 \le i \le n-1)$.
On writing $z=ru$ one observes that
$T_n(H, u)$ is at most
the number of solutions $a_0, \ldots, a_{n-1} \in \Z$, $z \in \N$, of
\begin{equation}
\label{eq1}
  z^2 = u \Disc{X^n + a_{n-1} X^{n-1} + \ldots +
  a_1 X + a_0}
\end{equation}
such that $|a_i| \le H \quad (0 \le i \le n-1)$,
and that~\eqref{aug6} and the conditions $|a_i| \le H \quad
(0 \le i \le n-1)$ force
$|r| \le H^{c_1}$, $|u| \le H^{c_2}$
for some constants $c_1, c_2>0$ only depending on $n$,
so $|z| \le H^c$ for some $c \ge 1$  
depending only on $n$. 
To bound the number of these solutions, we can now, in a completely
analogous way, follow the proof from~\cite[Section~5]{Diet2},
which deals with the special case $u=1$: First, fix
$a_2, \ldots, a_{n-1}$; there are $O(H^{n-2})$ choices.
By Lemma~\ref{eq:polyAA}, we may assume that
\begin{equation}
\label{eq:poly za1a2}
  z^2 - u \Disc{X^n + a_{n-1} X^{n-1} + \ldots +
  a_1 X + a_0}
\end{equation}  
as a polynomial in $z,a_1,a_0$ is absolutely irreducible. We can
therefore apply~\cite[Lemma~12]{Diet2}, and the same calculation
as in~\cite{Diet2} shows that there exist $J \ll H^{\sqrt{2}/2+\varepsilon}$ polynomials
$g_1, \ldots, g_J \in \Z[Z,A_1,A_0]$,   
such that each $g_j$ is coprime with the polynomial~\eqref{eq:poly za1a2}
and has degree bounded only in terms of $n$ and $\varepsilon$,
and every solution $(z,a_1,a_0)$ to~\eqref{eq1}
with
\begin{equation}
\label{eq:cond a1a2z}
|a_1|, |a_0| \le H  \mand |z| \le H^c,
\end{equation}
in addition
satisfies $g_j(z,a_1,a_0)=0$ for some $j \in \{1, \ldots, J\}$, apart possibly from
some exceptional set of solutions of cardinality
at most $H^{\sqrt{2}+o(1)}$.
So we have to consider $J$ systems of
two Diophantine equations, each consisting of~\eqref{eq1} and the equation 
$g_j(z,a_1,a_0)=0$ for some $j \in \{1, \ldots, J\}$. 

Fix any of those systems.
Then it is enough to show that there are at most
$H^{\sqrt{2}/2+o(1)}$ integer solutions  satisfying~\eqref{eq:cond a1a2z} 
to this system.  To this end, as in~\cite{Diet2}, we can
eliminate $z$ from the system, resulting in one Diophantine equation
$f_j(a_1, a_0)=0$, where $f_j \in \Z[A_1, A_0]$, which is a  non-zero rational polynomial
by the coprimality of $g_j$ and~\eqref{eq:poly za1a2}.
This can be factored over $\Q$, and as in~\cite {Diet2}, for each factor that is at least quadratic,
the bound of  Bombieri and Pila~\cite{BoPi} yields at most  $H^{1/2+o(1)}$ integer solutions with 
$|a_1|,|a_0| \le H$, which is more than satisfactory,
as from~\eqref{eq1}, for each pair $(a_1, a_0)$, we get
at most two solutions $z$.
The case of
linear factors is covered by Lemma~\ref{l3}, again yielding at most 
$H^{1/2+o(1)}$ solutions satisfying~\eqref{eq:cond a1a2z}. 

All together,
over all $J \ll H^{\sqrt{2}/2+\varepsilon}$ systems, and
considering the exceptional set, we obtain at most
$$
H^{\sqrt{2}+o(1)} + H^{\sqrt{2}/2+\varepsilon}
  H^{1/2+o(1)} =  H^{\sqrt{2}+o(1)}
$$ 
integer solutions with~\eqref{eq:cond a1a2z}, 
provided that $\varepsilon < \sqrt{2} -1$. 
Taking into account the
$O(H^{n-2})$ choices for $a_2, \ldots, a_n$ from the
beginning, we obtain
$$
T_n(H,u) \le H^{n-2+\sqrt{2}+o(1)},
$$ as required. 
\end{proof}

By  Lemma~\ref{lem:D/Delta}
we have
$$
N_n(H,\Delta)\le T_n(H,u),
$$
where $u$ is the square-free part of $\Delta$, and using Lemma~\ref{lem:det}, we now obtain the bound~\eqref{eq:nn-1 discr}.

\subsection{The bound~\eqref{eq:Any discr}: the square-sieve method}
\label{sec:SqSieve}

We recall the definitions of   $\cT_n(H,u)$ and 
$T_n(H,u)= \# \cT_n(H,u)$ from Section~\ref{sec:DetMethod}.

\begin{lemma}
\label{lem:sieve}
Uniformly over   square-free integers $u \ge 1$   we have the following estimate
$$
T_n(H,u) \ll  \begin{cases}H^{n-2n/(3n+3)}(\log H)^{(5n+1))/(3n+3)} & \text{if}\ n \ge 5.\\
 H^{n-n/(2n-1)}(\log H)^{(3n-2)/(2n-1)} & \text{if}\ n =3,4.
\end{cases} 
$$
\end{lemma}

\begin{proof} 
As before, by Lemma~\ref{lem:D/Delta}, $\Disc{f}$  and   $\Delta(f)$ 
have the same  square-free part, and thus $T_n(H,u)$ is 
the number of polynomials 
$f \in \cI_n(H)$  for which the square-free part of $\Disc{f}$  is $u$. 

We now apply the  square sieve of Heath-Brown~\cite{HB1} to 
the discriminants $\Disc{f}$ of polynomials  $f \in \cI_n(H)$.

Take now a real $z\ge 2$ and denote by $\cQ_z$ the set of all primes $p$ in the interval $(z,2z]$ and by $\pi(z,2z)$ the cardinality of this set, that is $\pi(z,2z)=\pi(2z)-\pi(z)$ where, as usual, $\pi(x)$ 
is the number of primes $p\le x$.

Clearly, for any  $f\in\cT_n(H,u)$ the product $u\Disc{f}$ is a perfect square, and thus,  
for a prime $p\ge 3$ we have
$$
\(\frac{u\Disc{f}}{p}\) = 1,
$$
unless $p\mid u\Disc{f}$, or equivalently  $p\mid \Disc{f}$ (as $u \mid  \Disc{f}$), 
in which case we have 
$$
\(\frac{u\Disc{f}}{p}\) = 0. 
$$

Note that the condition $f \in \cT_n(H,u) \subseteq \cI_n(H)$ 
automatically implies that $\Disc{f}\ne 0$.

Hence, for any $f\in\cT_n(H,u)$ we have
\begin{equation}
\label{eq:sieve}
\sum_{p\in\cQ_z}\(\frac{u\Disc{f}}{p}\)= \pi(z,2z) +  O\(\omega\(\Disc{f}\)\),
\end{equation} 
where $\omega(d)$ is the number of prime divisors of the integer
$d \ne 0$. 

Since $f\in\cI_n(H)$, we trivially have $\Disc{f} = H^{O(1)}$. 
Now, using the trivial bound $\omega(d) = O(\log d)$ 
and imposing the restriction 
\begin{equation}
\label{eq:z log H}
z\ge (\log H)^2,
\end{equation}
we see from the prime number theorem that 
\begin{equation}
\label{eq:z omega}
\pi(z,2z) +  O\(\omega\(\Disc{f}\)\)\ge \frac{1}{2} \pi(z,2z) 
\end{equation} 
provided that $H$ is large enough (certainly~\eqref{eq:z log H}
can be substantially relaxed, but this does not affect our result).  

Hence, from~\eqref{eq:sieve} and~\eqref{eq:z omega}
we conclude
$$
\frac{2}{\pi(z,2z)} \sum_{p\in\cQ_z}\(\frac{u\Disc{f}}{p}\) \ge 1.
$$
Squaring, summing over all $f\in\cT_n(H,u)$ and then expanding the summation 
to  all $f\in\cP_n(H)$, we obtain
\begin{align*}
T_n(H,u)&\le\frac{4}{\pi(z,2z)^2}\sum_{f\in\cT_n(H,u)}\left|\sum_{p\in\cQ_z}\(\frac{u\Disc{f}}{p}\)\right|^2\\
&\le \frac{4}{\pi(z,2z)^2}\sum_{f\in\cP_n(H)}\left|\sum_{p\in\cQ_z}\(\frac{u\Disc{f}}{p}\)\right|^2.
\end{align*}
Now, expanding the square and then changing the order of summation and using 
the multiplicativity of the Jacobi symbol,
we derive
\begin{equation}
\begin{split}
\label{eq:Cauchy T}
T_n(H,u) &\le \frac{4}{\pi(z,2z)^2}\sum_{f\in\cP_n(H)} \sum_{p,q\in\cQ_z}
\(\frac{u\Disc{f}}{pq}\)\\
& = \frac{4}{\pi(z,2z)^2}\sum_{p,q\in\cQ_z}\(\frac{u}{pq}\)
\sum_{f\in\cP_n(H)}\(\frac{\Disc{f}}{pq}\).
\end{split}
\end{equation}
Hence
$$
T_n(H,u)\ll  \frac{1}{\pi(z,2z)^2}\sum_{p,q\in\cQ_z}
\left|\sum_{f\in\cP_n(H)}\(\frac{\Disc{f}}{pq}\)\right|.
$$

If $n \ge 5$ we apply now the first bound of Lemma~\ref{lem:CharSum m H}  for the inner sum for $O(\pi(z,2z)^2)$ primes $p\ne q$ and the trivial bound $H^n$ for $\pi(z,2z)$ choices of primes $p=q$. Taking also into consideration that 
$$
\pi(z,2z) \gg \frac{z}{\log z}
$$
and $pq\le  4z^2$, we derive
\begin{align*}
T_n(H,u)& \ll z^{-1}H^n\log z+(H/z^2)^{n-1} z^{(3n+1)/2}
 \log z
+z^{(3n+1)/2}  (\log z)^n\\
& \ll z^{-1}H^n\log z+H^{n-1} z^{-n/2+5/2}
 \log z
+z^{(3n+1)/2}  (\log z)^n.
\end{align*} 
Choosing $z=H^{2n/(3n+3)}(\log H)^{-2(n-1)/(3n+3)}$, thus the condition~\eqref{eq:z log H} is
satisfied, we
obtain the desired bound.  

For $n =3,4$, we apply now the second bound of  Lemma~\ref{lem:CharSum m H}  for the inner sum for $O(\pi(z,2z)^2)$ primes $p\ne q$ and the trivial bound $H^n$ for $\pi(z,2z)$ choices of primes $p=q$. Taking also into consideration that 
$$
\pi(z,2z) \gg \frac{z}{\log z}
$$
and $pq\le  4z^2$, we derive
$$
T_n(H,u)\ll z^{-1}H^n\log z+H^{n-1} \log z
+z^{2n-2}(\log z)^n.
$$
Choosing $z=H^{n/(2n-1)}(\log H)^{-(n-1)/(2n-1)}$, thus the condition~\eqref{eq:z log H} is
satisfied, we conclude the proof. 
\end{proof}

As before, by  Lemma~\ref{lem:D/Delta}
we have
\begin{equation}
\label{eq: N and T}
N_n(H,\Delta)\le T_n(H,u),
\end{equation}
where $u$ is the square-free part of $\Delta$,  and using Lemma~\ref{lem:sieve}, we now obtain the bound~\eqref{eq:Any discr} and conclude the proof of Theorem~\ref{thm:NHDelta}. 

\section{Proof of Theorem~\ref{thm:MHD}}
\subsection{Bounds of mean of sums of Jacobi symbols}

We also make use of the following   bounds of character sums ``on
average'' over square-free moduli which are due to
Heath-Brown~\cite[Corollary 3.]{HB2}. In fact we only need a very special case 
of this result, which we present in the following form. 

\begin{lemma}
\label{lem:MeanVal} For all real positive numbers $D\ge 1$ and $Z\ge 1$, such that
$DZ\to \infty$, 
$$
\frac{1}{Z} \sum_{\substack{m \le Z\\ m~\text{odd square-free}}} \left |
\sum_{|\Delta| \le D} 
\(\frac{\Delta}{m}\)\right|^2 \le (DZ)^{o(1)}\sqrt{D(D/Z+1)}. $$
\end{lemma}

\subsection{Optimization of power sums}
We need the following technical result, see~\cite[Lemma~2.4]{GrKol}.

\begin{lemma}
\label{lem: Optim}
For $I,J\in\mathbb{N}$ let
$$
F(Z)=\sum_{i=1}^I A_i Z^{a_i}+\sum_{j=1}^JB_j Z^{-b_j},
$$
where $A_i,B_j,a_i$ and $b_j$ are positive for $1\le i\le I$ and $1\le j \le J$. Let $0\leq Z_1\leq Z_2$. Then there is some $Z\in [Z_1,Z_2]$ with
$$
F(Z)\ll \sum_{i=1}^I\sum_{j=1}^J
\(A_i^{b_j}B_j^{a_i}\)^{1/(a_i+b_j)}+\sum_{i=1}^I A_iZ_1^{a_i}+\sum_{j=1}^J B_jZ_2^{-b_j},
$$
where the implied constant depends only on $I$ and $J$.
\end{lemma}

\subsection{Concluding the proof}
Using~\eqref {eq: N and T} and also that 
 $$
\(\frac{u}{pq}\) = \(\frac{\Delta}{pq}\), 
$$
 where $u$ is the square-free part of $\Delta$, we can see that the bound~\eqref{eq:Cauchy T} implies 
$$
M_n(H,D) \le \frac{4}{\pi(z,2z)^2}\sum_{p,q\in\cQ_z} \sum_{|\Delta| \le D} \(\frac{\Delta}{pq}\)
\sum_{f\in\cP_n(H)}\(\frac{\Disc{f}}{pq}\).
$$
Hence
$$
M_n(H,D) \ll \frac{1}{\pi(z,2z)^2}\sum_{p,q\in\cQ_z} 
\left|\sum_{|\Delta| \le D} \(\frac{\Delta}{pq}\)\right|
\left| \sum_{f\in\cP_n(H)}\(\frac{\Disc{f}}{pq}\)\right|.
$$
Continuing as in Section~\eqref{sec:SqSieve}, and separating the contribution 
from the terms with $p=q$,  we obtain 
\begin{align*}
M_n(H,D) \ll z^{-1}& DH^n   \log z \\
&+  \frac{1}{\pi(z,2z)^2}\sum_{\substack{p,q\in\cQ_z\\p\ne q}} 
\left|\sum_{|\Delta| \le D} \(\frac{\Delta}{pq}\)\right|
\left| \sum_{f\in\cP_n(H)}\(\frac{\Disc{f}}{pq}\)\right|.
\end{align*}
 
 If $n \ge 5$ we apply now the first bound of Lemma~\ref{lem:CharSum m H} 
  for the inner sum and then the bound of Lemma~\ref{lem:MeanVal}, and thus derive 
  (after replacing all power logarithms with $H^{o(1)}$ 
\begin{align*}
M_n(H,D) \ll z^{-1}& DH^{n+o(1)}    \\
&+  H^{o(1)}  \(\(H/z^2\)^{n-1}  +1\)  z^{(3n+1)/2}\sqrt{D  (D/z^2+ 1)}
\end{align*}
After some trivial manipulations,  we obtain 
\begin{equation}
\label{eq: M and M}
M_n(H,D) \ll  H^{o(1)} \cM , 
  \end{equation}
where 
\begin{align*}
\cM = z^{-1} DH^{n}   + z^{-(n-3)/2} DH^{n-1} & + z^{-(n-5)/2} D^{1/2} H^{n-1}\\
&  \quad +   z^{(3n-1)/2} D +  z^{(3n+1)/2} D^{1/2} 
\end{align*} 

Since we obviously have  $z^{-1} DH^{n}  \ge  z^{-(n-3)/2} DH^{n-1}$ 
we can simplify the above bound as 
\begin{align*}
\cM & \ll  z^{-1} DH^{n}  + z^{-(n-5)/2} D^{1/2} H^{n-1} +   z^{(3n-1)/2} D +  z^{(3n+1)/2} D^{1/2} \\
&= \(z^{-1} D^{1/2} H^{n}  + z^{-(n-5)/2}   H^{n-1} +   z^{(3n-1)/2} D^{1/2}  +  z^{(3n+1)/2}\) D^{1/2}.
\end{align*}
We now apply Lemma~\ref{lem: Optim} with $I = J=2$, $Z=(DH)^{100} $, $Z_1=(\log H)^2$ (see~\eqref{eq:z log H}), $Z_2=(DH)^{100}$ and parameters
\begin{align*}
& (A_1,a_1)     = \(D^{1/2} , (3n-1)/2\) , \qquad &(A_2,a_2)     = \(1 , (3n+1)/2\) , \quad \\
&(B_1,b_1) = (D^{1/2} H^{n},1), \qquad  & (B_2,b_2) = ( H^{n-1}, (n-5)/2).
\end{align*}
We now compute
\begin{align*}
\( A_1^{b_1} B_1^{a_1}\)^{1/(a_1+b_1)}  &= \(D^{1/2} \(D^{1/2} H^{n}\)^{(3n-1)/2} \)^{2/(3n+1)}\\
 &= D^{1/2}  H^{n(3n-1)/(3n+1)},\\
 \( A_1^{b_2} B_2^{a_1}\)^{1/(a_1+b_2)}  &= \(D^{ (n-5)/4} \( H^{n-1}\)^{(3n-1)/2} \)^{1/(2n-3)}\\
 &=D^{ (n-5)/(8n-12)} H^{(n-1)(3n-1)/(4n-6)},\\
 \( A_2^{b_1} B_1^{a_2}\)^{1/(a_2+b_2)}  &= \(  \(D^{1/2} H^{n}\)^{(3n+1)/2} \)^{2/(3n+3)}\\
 &= D^{(3n+1)/(6n+6)}  H^{n(3n+1)/(3n+3)},\\
  \( A_2^{b_2} B_2^{a_2}\)^{1/(a_2+b_2)}  &= \(  \( H^{n-1}\)^{(3n+1)/2} \)^{1/(2n-2)}\\
 &=    H^{(3n+1)/4}.
\end{align*}
Certainly the contribution from the terms involving $Z_1$ and $Z_2$ is negligible. 
We also note that for $n \ge 5$ we have
$$
 D^{1/2}  H^{n(3n-1)/(3n+1)} \ge  H^{n(3n-1)/(3n+1)} \ge H^{(3n+1)/4}. 
 $$
  Hence the last term $H^{(3n+1)/4}$ can be omitted.
Furthermore,  for $n \ge 5$ we also have
 $$
\frac{n-5}{8n-12} \le \frac{3n+1}{6n+6} \mand  \frac{(n-1)(3n-1)}{4n-6} \le \frac{n(3n+1)}{3n+3}. 
 $$
  Hence the second term $D^{ (n-5)/(8n-12)} H^{(n-1)(3n-1)/(4n-6)}$ can be omitted too.
Thus we obtain 
$$
\cM   \ll    D   H^{n(3n-1)/(3n+1)}   +   D^{(3n+1)/(6n+6)+1/2}  H^{n(3n+1)/(3n+3)}  .
$$
Recalling~\eqref{eq: M and M} we obtain 
$$
M_n(H,D) \le   D  H^{n(3n-1)/(3n+1)+o(1)}   + D^{(3n+2)/(3n+3)}  H^{n(3n+1)/(3n+3)+o(1)} . 
$$

We now observe that the second  term improved the trivial bound
$M_n(H,D) \ll H^n$ only for $D \le H^{2n/(3n+2)}$, in which case the second  term also dominates 
the first term as 
$$
D H^{n(3n-1)/(3n+1)}  \le  D^{(3n+2)/(3n+3)} H^{n(3n+1)/(3n+3)} 
$$
is equivalent to $D \le  H^{4n/(3n+1)}$. 
The desired result now follows.

\section{Proof of Theorem~\ref{thm:trinom}}

Write $H=A+B+C+D$. There are $O(A)$ choices for $a$, so it suffices to
show that for fixed $a \in [C, C+A]$ there are at most $H^{o(1)}$ solutions
$(b,r) \in \Z^2$, $b \in [D, D+B]$ to the equation
$$
  ur^2-n^n b^{n-1} = (n-1)^{n-1} a^n.
$$
As $n \equiv 1 \pmod 4$, substituting $t=b^{(n-1)/2}$, it is enough to
uniformly in $a$ bound the number of $r, t \in \Z$,
$|r|, |t| \ll  H^{n/2}$, such that
\begin{equation}
\label{tag}
  ur^2-n^nt^2 = (n-1)^{n-1}a^n.
\end{equation}
Note that $(n-1)^{n-1} a^n \ne 0$ as $n>1$ and $a \in [C, C+A]$ where
$C \ge 1$. If $-un^n$ is a square in $\Z$, then we can
factor the left hand side of~\eqref{tag} and use the divisor function
estimate $\tau(m)= m^{o(1)}$ for all $m \in \Z
\backslash \{0\}$ to see that~\eqref{tag} has at most $H^{o(1)}$
solutions $r, t \in \Z$. If $-un^n$ is no square, then the left hand
side of~\eqref{tag} is a Pellian type equation and though~\eqref{tag} has possibly infinitely many solutions $r, c \in \Z$,
the number of solutions such that $|r|, |t| \ll H^{n/2}$ by a familiar
result can be bounded by   $H^{o(1)}$, see, for example,~\cite[Lemma~3]{KonShp} 
for arbitrary quadratic polynomials.

\section{Concluding Comments}
\label{sec:comm}

\subsection{Discriminants of splitting fields of polynomials} 
As mentioned in Remark~\ref{rem:disc split}, it is certainly interesting to count the discriminants 
of splitting fields of polynomials $f \in \cI_n(H)$. 
Unfortunately our basic tool, Lemma~\ref{lem:D/Delta},
does not generalise to the discriminants of these fields.
Motivated by this and also by an apparently terminological 
oversight at the beginning of~\cite[Section~1]{ABZ}
(where $\Delta(f)$ is called the discriminant of the 
splitting field of $f$), we give two examples showing that 
such a direct analogue of  Lemma~\ref{lem:D/Delta} is
false.

In particular, for the polynomial $f(X) = X^4-2$ it is easy to check
that the splitting field $L$ of $f$ over $\Q$ is given by
$L=\Q(\sqrt[4]{2}, i)$ and that $|L:\Q|=8$.
Further, it is not hard to see that $\sqrt[4]{2}(1+2i)$ satisfies the
equation $F(\sqrt[4]{2}(1+2i))=0$, where $F(X)$ is the degree $8$
polynomial $F(X)=X^8+28 X^4+2500$ which is irreducible in
$\Q[X]$. Hence $L=\Q\(\sqrt[4]{2}(1+2i)\)$. Using the discriminant formula
for trinomials, one finds that $\Disc{f}=-2^{11}$ and
$\Disc{F}=2^{62} \cdot 3^8 \cdot 5^{12}$. As $\Disc{f}<0$ and
$\Disc{F}>0$, by Lemma~\ref{lem:D/Delta} the ratio of
$\Disc{f}$ and the discriminant $\Delta$ of $L$ is not a rational
square (in fact, using for example {\sl Sage}, one can check that
$\Delta=2^{24}$, so $\Delta/\Disc{f}=-2^{13}$; see also
Global Number Field 8.0.16777216.2 in~\cite{lmfdb}).

A slightly more complicated non-binomial example is given by 
the polynomial $f(X) = X^4-X-1$.
{\sl Magma} computes the defining polynomial of the
splitting field of $f$  as
\begin{align*}
F(X) & =X^{24}+ 90X^{21} - 70X^{20} + 5695X^{18} - 18690X^{17} + 34895X^{16}\\
& \qquad +
 225900X^{15} - 1544060X^{14} + 3867780X^{13} + 18840027X^{12} \\
& \qquad- 62876100X^{11} +
 228621050X^{10} - 222888810X^{9} \\
& \qquad+ 999415025X^{8} + 9907474500X^{7}-
 24575577355X^{6} \\
& \qquad+ 34467394920X^{5} + 232838692457X^{4}- 705674357100X^{3}\\
& \qquad +
 2030693398335X^{2} - 2155371295770X + 1779496656001.
\end{align*}
Since $\Disc{f} = 283$ and 
\begin{align*}
 \Disc{F} & = 2^{144} \cdot {3}^{24} \cdot {17}^{8} \cdot {37}^{4} \cdot {73}^{2} 
 \cdot {83}^{2} \cdot {101}^{2} \cdot {181}^{2}
  \cdot {227}^{2} \cdot {283}^{12}\\
  & \qquad  \cdot {359}^{4} \cdot {8867}^{8} \cdot {9473}^{2} \cdot {47777}^{4} \cdot {1271971}^{2} \cdot {1660069}^{4} \\
  & \qquad \cdot {970293859}^{2} \cdot {4552394491}^
  {2}  \cdot {857054278934851321}^{2}\\
  & \qquad  \cdot {1521484680115687561}^{2},
\end{align*}
the presence of the even power of $283$ in the prime number factorisation 
of $ \Disc{F}$ and Lemma~\ref{lem:D/Delta} 
show that the ratio of  $\Disc{f}$ and the discriminant of the splitting field
is not a rational square.

We note that both approaches, via the determinant method and via the square sieve are 
flexible enough to admit   
several variations in the way we count polynomials. For example,
one can fix some of the coefficients,  or make them run in a non-cubic box, 
$[-H_0, H_0]\times \ldots \times [-H_{n-1},H_{n-1}]$, or move the boxes away from the origin, 
as in Section~\ref{sec:discr trinom}. 

\subsection{Discriminants of polynomials} 
It is also natural to ask about the number $D_n(H)$ of  distinct discriminants that are generated by all polynomials from 
$\cI_n(H)$.  It  is reasonable   to expect $D_n(H) = H^{n+o(1)}$, however 
this question seems to be open.
We briefly note that trinomials immediately imply 
$D_n(H) \gg H^{2}$.
Indeed, we consider the discriminants
$$
\Disc{X^n + aX - b}  = (-1)^{(n-1)(n+2)/2}((n-1)^{n-1} a^n+n^n b^{n-1})
$$
of trinomials $X^n + aX - b$ (see for example,~\cite[Theorem~2]{Swan}) with 
$$ H/2 \le a \le H \mand  1 \le b \le \frac{H}{3n}
$$
with the additional condition 
$$
a\equiv 0 \pmod 2 \mand b\equiv 2 \pmod 4
$$
to guarantee the irreducibility by the Eisenstein criterion.  
We claim all such pairs $(a,b)$ generate distinct discriminants. 
Indeed,  if 
$$(n-1)^{n-1} a_1^n+ n^n b_1^{n-1} = (n-1)^{n-1} a_2^n+ n^n b_2^{n-1} 
$$
then for $a_1=a_2$ we also have $b_1 = b_2$. So we can now assume that 
$a_1 > a_2$. 
In this case  we obtain 
\begin{align*}
(n-1)^{n-1} a_1^n -  (n-1)^{n-1} a_2^n & \ge (n-1)^{n-1} a_1^n -  (n-1)^{n-1} (a_1-1)^n\\
&\ge  n (n-1)^{n-1} (H/2)^{n-1}  +O(H^{n-2}) \\
& = 2^{-n+1} n (n-1)^{n-1} H^{n-1}  +O(H^{n-2})
\end{align*}
while 
$$
n^n b_2^{n-1} - n^n b_1^{n-1}  \le  n^n b_2^{n-1} \le 3^{-n+1} n H^{n-1}
$$
which is impossible for a sufficiently large $H$. 

Unfortunately, this argument does not give the lower bound $H^{2+o(1)}$ for the number of distinct discriminants of fields generated by roots of polynomials in $\cI_n(f)$, improving Corollary~\ref{cor:distinct disc}, since having distinct discriminants of polynomials does not imply necessarily distinct discriminants of fields.

Finally, we note that our methods can also be used to investigate the discriminants of the fields 
generated by some other special families of polynomials.  For example,  one of such families is 
given by quadrinomials $X^n + aX^2 +bX + c$ 
for the discriminant of which an explicit formula has been given by Otake and Shaska~\cite{OtSh}. 

\section{Appendix}
\label{sec:app}

\subsection{Preliminary discussion} 
We use this opportunity to fix an error in~\cite{Diet2}. 
Namely~\cite[Lemmas~5 and~6]{Diet2}  (and consequently~\cite[Lemma~8]{Diet2})
there are not correct as stated if the degree $n$ is of the form
$n=\newu^2$ or $n=\newu^2+1$ for some odd $m$,
and therefore~\cite[Lemma~8]{Diet2} cannot always
be directly applied in these cases as well.  This does not affect the
main results~\cite[Theorems~1 and~2]{Diet2}
in these cases, so let us
quickly explain how to amend the proof:

\subsection{The case of $n=\newu^2$}
If $n=\newu^2$
for odd $\newu$, then we can directly handle the contribution of $a_n$ 
such that $z^2-\Delta(a_1, \ldots, a_n)$ is reducible:
\cite[Lemma~6]{Diet2} as well as
\cite[Lemma~5]{Diet2} in the case of $c_1 \ne 0$ are still
correct. As a substitute for~\cite[Lemma~5]{Diet2} 
for $c_1=0$, we can use~\cite[Satz 1]{Hering0}
(see also~\cite[Section~1]{Hering}).
The latter result shows that 
for fixed $a_1, \ldots, a_{n-2} \in \Z$, there are,
uniformly in $a_1, \ldots, a_{n-2}$, only finitely many
rational specialisations for
$a_{n-1}$, for which the resulting polynomial
$f(X)=X^n+a_1 X^{n-1}+\ldots+a_n$, regarded as a polynomial in
$\Q(a_n)[X]$, does not have
Galois group $S_n$ over the rational function field $\Q(a_n)$.
Only in these cases $z^2-\Delta(a_1, \ldots,
a_n)$, as a polynomial in $z$ and $a_n$,
can be reducible over $\Q$, since otherwise having Galois group
$S_n$ over $\Q(a_n)$ excludes the
possibility that the discriminant $\Delta(a_1, \ldots, a_n)$ is a
square in $\Q(a_n)$. Therefore
there can be only $O(1)$ many exceptional `bad planes' given 
by~\cite[Equation~(5)]{Diet2}, for which the bound
in~\cite[Lemma~8]{Diet2} does not hold true. Just using the trivial bound
$O(H)$ for the number of solutions in these cases
instead of the bound provided by~\cite[Lemma~8]{Diet2}
is acceptable, as  
 the resulting bound of $O(H^{n-2})$ (for
fixing $a_1, \ldots, a_{n-2}$) times $O(1)$ (for the number of exceptional
`bad planes')
times $O(H)$ (trivially bounding the solutions instead
of using the bound from~\cite[ Lemma~8]{Diet2})
is certainly $H^{n-2+\sqrt{2}+o(1)}$. This fixes the error
for $n=\newu^2$ and odd $\newu$.

 \subsection{The case of $n=\newu^2+1$}
If $n=\newu^2+1$ for odd $\newu \ge 3$, then $n$ cannot
be divisible by $3$.   
In this case, at the outset instead of
fixing $n-2$ coefficients $a_1, \ldots, a_{n-2}$ we fix $n-2$ coefficients $a_1, \ldots, a_{n-4}, a_{n-2},
a_{n-1}$ instead.
As a substitute for~\cite[Lemma~5]{Diet2} in the case of
$c_1 \ne 0$ we prove the following result. 
\begin{lemma}
Let $\newu \ge 3$ be an odd integer, and let $n=\newu^2+1$. Further, let
$a_1, \ldots, a_{n-4}, a_{n-2}, a_{n-1}$ be fixed integers, and
let $c_1, c_2 \in \Q$ with $c_1 \ne 0$. Then the polynomial
$$
  z^2-\Delta(a_1, \ldots, a_{n-4}, c_1 a_n + c_2, a_{n-2}, a_{n-1}, a_n)
$$
is irreducible in $\Q[z, a_n]$.
\end{lemma}
\begin{proof}
We use the observation that for fixed
$a_1, \ldots, a_{n-4}, a_{n-2}, a_{n-1}$, the discriminant
$\Delta(a_{n-3}, a_n)=\Delta(a_1, \ldots, a_n)$ as a polynomial
in $a_{n-3}$ and $a_n$ is of the form  
\begin{equation}
\label{master}
  \Delta(a_{n-3}, a_n)=(n-3)^{n-3}3^3 a_{n-3}^n a_n^2 +
  \Phi(a_{n-3}, a_n),
\end{equation}
where $\Phi$ has total degree strictly less than $n+2$. The proof is
analogous to that of~\cite[Lemma~4]{Diet2}, using the fact that
$\Delta(a_1, \ldots, a_n)$ is a weighted-homogeneous polynomial in the
$a_i$, each $a_i$ having weight $i$, and the total weight of
$\Delta(a_1, \ldots, a_n)$ is $n(n-1)$. Therefore, 
for fixed $a_1, \ldots, a_{n-4}, a_{n-2}, a_{n-1}$
any monomial
$a_{n-3}^\alpha a_n^\beta$ occurring in $\Delta(a_1, \ldots, a_n)$
satisfies
\begin{equation}
\label{yacht}
  (n-3)\alpha+n\beta \le n(n-1).
\end{equation}
For $\alpha=n$ and $\beta=2$ the left hand side of~\eqref{yacht} just
equals $n(n-1)$, whence the monomial
$\delta_n a_{n-3}^n a_n^2$ occurs in
$\Delta(a_{n-3}, a_n)$, with a constant $\delta_n$ only depending on $n$; 
note that we do not yet know whether $\delta_n \ne 0$.
To establish~\eqref{master} it is therefore enough to check that this is
the only
solution of~\eqref{yacht} with $\alpha+\beta \ge n+2$, and then to
evaluate $\delta_n$. If $\alpha+\beta \ge n+2$ and
$\beta \ge 3$, then
\begin{align*}
(n-3)\alpha+n\beta & \ge (n-3)(n+2-\beta)+n\beta\\
  & = (n-3)(n+2)+3\beta\\
  & = n(n-1)-6+3\beta  \ge n(n-1)+3.
\end{align*}
If $\beta \le 1$, then $\alpha+\beta \ge n+2$ gives
$\alpha>n$, which is impossible, because
the maximum power of any $a_i$ occurring in any monomial of
$\Delta(a_1, \ldots, a_n)$ is at most $n$. The latter is
easily checked by writing the discriminant $\Delta(a_1, \ldots, a_n)$
in the form
$$
  \Delta(a_1, \ldots, a_n)=(-1)^{n(n-1)/2} \Res{f, f'},
$$
(see the formula~\eqref{eq:D vs R}),
where $f=X^n+a_1 X^{n-1} + \ldots + a_n$, expressing the resultant
$\Res{f,f'}$ of $f$ and its derivative $f'$ by the Sylvester 
formula as a certain determinant in $a_1, \ldots, a_n$, and checking
that each $a_i$ occurs in at most $n$ columns. Hence
$$
  \Delta(a_{n-3}, a_n)=\delta_n a_{n-3}^n a_n^2+\Phi(a_{n-3}, a_n),
$$
where $\Phi$ has total degree less than $n+2$. To determine the
value of $\delta_n$ (which only depends on $n$ as remarked above),
we observe that,
as $n$ is coprime to $3$, 
the trinomial $X^n+aX^3+b$ has discriminant
$$
  (-1)^{n(n-1)/2} b^2 (n^n b^{n-3} + (-1)^{n+1} (n-3)^{n-3} 3^3 a^n)
$$ 
(see, for example,~\cite[Theorem 2]{Swan}), which immediately yields
$$
  \delta_n=(-1)^{n(n-1)/2+n+1} (n-3)^{n-3} 3^3 = (n-3)^{n-3} 3^3
$$
as $n=\newu^2+1 \equiv 2 \pmod 4$.
Having established~\eqref{master}, we see that for $n$ coprime to $3$ the number
$(n-3)^{n-3}3^3$ cannot be a square, whence
\begin{align*}
  & z^2-\Delta(a_1, \ldots, a_{n-4}, c_1a_n+c_2, a_{n-2}, a_{n-1}, a_n)\\
  & = z^2 - (n-3)^{n-3}3^3 c_1^n a_n^{n+2} + O(a_n^{n+1})
\end{align*}
is irreducible in $\Q[z, a_n]$.
\end{proof}
The special cases that $a_{n-3}$ or $a_n$ are being fixed
(substitutes for the analogues of~\cite[Lemma~5]{Diet2} where
$c_1=0$, and~\cite[Lemma~6]{Diet2}, respectively)
 can be handled
as above by the result of Hering~\cite{Hering0}, again using that
$n$ is coprime to $3$.
The argument can then be
finished as above,
 using the main result of~\cite{Smith} 
 instead of~\cite[ Lemma~10]{Diet2} to see that for $n$ coprime to $3$ the
polynomial $X^n+aX^3+b$ has Galois group
$S_n$ over any function field $K(a,b)$ where $K$ is any field of
characteristic zero.

\section*{Acknowledgement}

The authors are grateful to Nicholas Katz for valuable discussions regarding several issues about  discriminants of number fields 
and also for providing the second example of Section~\ref{sec:comm}. 
The authors also  would like to thank the referee for the 
careful reading of the paper and several valuable suggestions
improving the exposition of the paper.

During the preparation of this work, A.~O. was supported by the
ARC Grant DP180100201
and I.~S.   was  supported   by the ARC Grant DP170100786.


\begin{thebibliography}{9999}


\bibitem{AGLLSZ} T. C. Anderson, A. Gafni, R. J. Lemke Oliver, D. Lowry-Duda, G. Shakan, and R. Zhang, 
`Quantitative Hilbert irreducibility and almost prime values of polynomial discriminants',
{\it Intern. Math. Res. Notices\/},  (to appear). 
 


\bibitem{ABZ} A. Ash, J. Brakenhoff and T. Zarrabi,
`Equality of polynomial and field discriminants',
{\it Experim.  Math.\/}, {\bf 16} (2007),  367--374.

\bibitem{BBP} K. Belabas, M. Bhargava and C. Pomerance, 
`Error estimates for the Davenport--Heilbronn theorems',
{\it  Duke Math. J.\/}, {\bf 153} (2010), 173--210. 


\bibitem{Bha}
M. Bhargava,
`Galois groups of random integer polynomials and van der Waerden's Conjecture',
{\it  Preprint\/}, 2021 (available from \url{http://arxiv.org/abs/2111.06507}).

\bibitem{BSW}
M. Bhargava, A. Shankar and X. Wang, 
`Squarefree values of polynomial discriminants~I',
{\it  Invent.  Math.\/},  (to appear). 

\bibitem{BiLe} P.-Y. Bienvenu and T. H. L{\^e}, 
`Linear and quadratic uniformity of the M{\"o}bius function over $\Fq[t]$', 
{\it Mathematika\/}, {\bf 65} (2019), 505--529.

\bibitem{BoPi}
 E. Bombieri and J. Pila, 
`The number of integral points on arcs and ovals',
{\it  Duke Math. J.\/}, {\bf 59} (1989),  337--357. 

\bibitem{CaRu}
D. Carmon and Z. Rudnick, `The autocorrelation of the M{\"o}bius function 
and Chowla's conjecture for the rational function field', 
{\it Quart. J. Math.\/}, {\bf 65} (2014), 53--61. 


\bibitem{Chela} 
R. Chela, `Reducible polynomials', {\it J. London Math. Soc.\/}, 
{\bf 38} (1963), 183--188.


\bibitem{Dalen} K. Dalen,
`On a theorem of Stickelberger',
{\it Math. Scand.\/}, {\bf 3} (1955), 124--126.

\bibitem{Diet1} R.~Dietmann,
`On the distribution of Galois groups',
{\it Mathematika\/}, {\bf 58} (2012),  35--44.

\bibitem{Diet2} R.~Dietmann,
`Probabilistic Galois theory', {\it Bull. London Math. Soc.},
{\bf 45} (2013), 453--462.

\bibitem{ElVe} J. S. Ellenberg and A. Venkatesh, 
`The number of extensions of a number field with fixed degree and bounded discriminant',
{\it Ann. Math.\/}, {\bf 163} (2006), 723--741.

  \bibitem{GrKol}  S. W. Graham and  G. Kolesnik,
{\it Van der Corput's method of exponential sums\/}, Cambridge Univ.
Press, 1991.
 
\bibitem{HB1} D. R. Heath-Brown,
`The square sieve and consecutive squarefree numbers',
{\it Math. Ann.\/}, {\bf 266} (1984),  251--259.

\bibitem{HB2} D. R. Heath-Brown,
`A mean value estimate for real character sums', {\it Acta Arith.\/}
{\bf  72} (1995), 235--275.

\bibitem{HB3} D. R. Heath-Brown,
`The density of rational points on curves and surfaces',
{\it Ann. Math.}, {\bf 155} (2002), 553--595.

\bibitem{Hering0} H. Hering,
`Seltenheit der Gleichungen mit Affekt bei linearem Parameter',
{\it Math. Ann.}, {\bf 186} (1970), 263--270.

\bibitem{Hering} H. Hering,
`\"Uber Koeffizientenbeschr\"ankungen affektloser
Gleichungen', {\it Math. Ann.}, {\bf 195} (1972), 121--136.

\bibitem{ILOSS} R. Ibarra, H. Lembeck, M. Ozaslan, H. Smith and K. Stange, 
`Monogenic fields arising from trinomials',
{\it  Involve\/},  (to appear). 

\bibitem{IwKow} H. Iwaniec and E. Kowalski,
{\it Analytic number theory\/}, Amer.  Math.  Soc.,
Providence, RI, 2004.

\bibitem{Jones1}
L. Jones, 
`A brief note on some infinite families of monogenic polynomials',
{\it  Bull. Aust. Math. Soc.\/}, {\bf  100} (2019), 239--244.

\bibitem{Jones2}
L. Jones, 
`Monogenic polynomials with non-squarefree discriminant',
{\it  Proc. Amer. Math. Soc.\/}, {\bf  148} (2020), 1527--1533. 

\bibitem{JoWh}
L. Jones and D. White, 
`Monogenic trinomials with non-squarefree discriminant',
{\it  Preprint\/}, 2019 (available from \url{http://arxiv.org/abs/1908.07947}).


\bibitem{Katz} N. Katz, `Estimates for nonsingular mixed character sums',
{\it International Mathematics Research Notices\/}, 
\textbf{2007} (2007),  Article ID rnm069, 1--19.



 
\bibitem{Kedl}  K. S. Kedlaya,
`A construction of polynomials with squarefree discriminants', 
{\it Proc. Amer. Math. Soc.\/},  {\bf 140} (2012), 3025--3033.

\bibitem{KonShp} S. V. Konyagin and I. E. Shparlinski, 
`On  convex hull of  points on modular hyperbolas', 
{\it Moscow J. Comb. and Number Theory\/}, {\bf 1} (2011), 43--51.


\bibitem{Lang}
S. Lang, {\it Algebraic number theory\/},
Springer, Berlin, 1970.

\bibitem{LaRo}  E. Larson and L. Rolen, `Upper bounds for the number of number fields with alternating Galois group', 
{\it Proc. Amer. Math. Soc.\/},  {\bf 141} (2013), 499--503.

\bibitem{lmfdb} LMFDB - The L-functions and modular forms database,
\url{http://www.lmfdb.org/NumberField}.

\bibitem{MMS} A. Mukhopadhyay, M. R. Murty and K.~Srinivas,
`Counting squarefree discriminants of trinomials under $abc$',
{\it  Proc. Amer. Math. Soc.\/}, {\bf 137}  (2009), 
3219--3226. 


\bibitem{OtSh} S. Otake and T. Shaska, `On the discriminant of certain quadrinomials',
{\it Contemp. Math.\/}, vol.~724, 2019, Amer. Math. Soc., 55--72. 

\bibitem{Poon}
B. Poonen, `Squarefree values of multivariable polynomials',
{\it  Duke Math. J.\/}, {\bf 118} (2003), 353--373.

\bibitem{Por} S. Porritt, `A note on exponential-M{\"o}bius sums over $\Fq[t]$',
{\it Finite Fields  Appl.\/}, {\bf 51}  (2018), 298--305.
  

\bibitem{Ro-Le} A. Rojas-Le{\'o}n, `Estimates for singular multiplicative character sums',
{\it International Mathematics Research Notices\/}, 
\textbf{2005} (2005),  1221--1234.

\bibitem{Ros} M. Rosen,
{\it Number theory in function fields\/},
Springer,  Berlin,  2002.

\bibitem{S} P. Salberger,  `Counting rational points on projective
varieties', {\it Preprint\/}.

\bibitem{Shp0} I. E. Shparlinski,
`Distribution of primitive and irreducible polynomials modulo a prime', 
{\it Diskret. Mat.\/}, {\bf  1} (1989), no. 1, 117--124 (in Russian);  translation in
{\it Discrete Math. Appl.\/}, {\bf  1} (1991),  59--67. 

\bibitem{Shp1} I. E. Shparlinski,
`On  quadratic fields generated by discriminants of
irreducible  trinomials', 
{\it Proc. Amer. Math. Soc.\/}, {\bf 138} (2010), 125--132. 

\bibitem{Shp2} I. E. Shparlinski,
`Distribution of polynomial discriminants modulo a prime', 
{\it Arch. Math.\/}, {\bf 105}  (2015), 251--259.

\bibitem{Smith} J. H. Smith,
`General trinomials having symmetric Galois group',
{\it Proc. Amer. Math. Soc.}, {\bf 63} (1977), 208--212.

\bibitem{St} L. Stickelberger, `{\"U}ber eine neue Eigenschaft der Diskriminanten 
algebraischer Zahlk{\"o}rper', {\it Verh. 1 Internat. Math. Kongresses, 1897\/}, 
Leipzig, 1898, 182--193. 


\bibitem{Swan} R. G. Swan, `Factorization of polynomials over finite fields', 
{\it Pacific J. Math.\/}, {\bf 12} (1962), 1099--1106.


\bibitem{Zyw} D. Zywina, `Hilbert's irreducibility theorem and the larger sieve',
{\it  Preprint\/}, 2010 (available from \url{http://arxiv.org/abs/1011.6465}).


\end{thebibliography}
\end{document}